\documentclass{article}[12 pt]
\usepackage{geometry}
\usepackage[all]{xy}
\usepackage{mathrsfs}
\usepackage{amsthm,amsmath,amssymb,enumitem}
\usepackage{setspace}
\usepackage[all]{xy}
\usepackage{amsmath,color}
\usepackage{mathrsfs}
\usepackage{amsmath}
\usepackage{geometry}
\usepackage{authblk}

\newdir{ >}{{}*!/-5pt/@{>}}
\theoremstyle{plain}
\newtheorem{thm}{Theorem}[section]
\theoremstyle{definition}
\newtheorem{defn}{Definition}[section]
\newtheorem{prop}{Proposition}[section]
\newtheorem{cor}{Corollary}[section]
\newtheorem{lemma}{Lemma}[section]
\newtheorem{eg}{Example}[section]
\newtheorem{rmk}{Remark}[section]

\DeclareMathOperator{\Ass}{Ass}
\DeclareMathOperator{\Spec}{Spec}

\DeclareMathOperator{\Supp}{Supp}
\DeclareMathOperator{\ann}{ann}
\DeclareMathOperator{\Ord}{Ord}
\DeclareMathOperator{\rk}{rank}

\title{}
\begin{document}
\title{On Coprimary Filtrations}
\author{Yao Li}
\maketitle
\section{Introduction}
The coprimary filtration is a basic construction in commutative algebra (see \cite[Theorem 6.4]{M1989}).  Here we review the outcome: Let $A$ be a Noetherian ring, $M$ be a non-zero finitely generated $A$-module. Then there exists a chain of submodules of $M$,
\[0=M_0\subset M_1\subset \cdots \subset M_n=M,\]
such that $M_i/M_{i-1}\cong A/P_i$, where $P_i\in \Spec(A)$, $i\in\{1,\cdots,n\}$. In this result, $P_i$ may not belong to $\Ass(M)$. This result could be extended in algebraic geometry in the setting of Noetherian scheme (see \cite[Corollary 3.2.8]{G1965}). Let $X$ be a Noetherian scheme, $\mathcal F$ be a non-zero coherent sheaf on $X$. Then there exists a filtration of coherent sheaves on $X$,
\[0=\mathcal F_0\subset \mathcal F_1\subset \cdots \subset \mathcal F_n=\mathcal F,\]
such that $\mathcal F_i/\mathcal F_{i-1}$ is coprimary and $\Ass(\mathcal F_i/\mathcal F_{i-1})\subseteq \Ass(\mathcal F)$, where $i\in\{1,\cdots, n\}$. Recently, a new result has been discovered by Chen and Jeannin \cite{CJ2023} by using the idea from Harder-Narasimhan theory. \\

The notion of Harder-Narasimhan filtration was first introduced by Harder and Narasimhan \cite{HN1974} in the setting of vector bundles on a non-singular projective curve. Let $k$ be an algebraically closed field and $C$ be a non-singular projective curve over $k$. Then any non-zero vector bundle $E$ on $C$ admits a canonical filtration by vector subbundles,
\[0=E_0\subset E_1\subset\ldots\subset E_n=E,\]
where the subquotients $E_i/E_{i-1}$ are semi-stable vector bundles with strictly decreasing slopes. Curiously, analogous constructions have also been  found in many other branches of mathematics which motivate various categorical constructions of Harder-Narasimhan filtration, see for example \cite[\S II.2]{L1997}, \cite{R1997}, \cite{A2009}, \cite{C2010}, \cite{L2024}.
In Li's work \cite{L2024}, he developed the Harder-Narasimhan theory via slope functions valued in a totally ordered set, rather than employing degree functions and rank functions. Later, the chain
conditions of \cite{L2024} are relaxed by Chen and Jeannin \cite{CJ2023}. By directly applying the Harder-Narasimhan theory, they obtained the following result: Let $A$ be a Noetherian ring, $M$ be a non-zero finite $A$-module and $(\Ass(M),<)$ be a totally ordered set which is a linear extension of the partially ordered set $(\Ass(M),\subset)$. Denote $\Ass(M)$ by $\{\mathfrak{p}_1,\cdots, \mathfrak{p}_n\}$, where $\mathfrak{p}_1>\mathfrak{p}_2>\cdots>\mathfrak{p}_n$. Then there exists a unique chain of submodules of $M$,
\[0=M_0\subset M_1\subset \cdots \subset M_n=M,\]
such that $\Ass(M_i/M_{i-1})=\{\mathfrak{p}_i\}$, where $i\in \{1,\cdots, n\}$. In this result, they construct the destabilizing object $M_1$ by using Noetherian condition and $\mu_A$-descending chain condition. In this article, we prove the existence and uniqueness of coprimary filtration of modules (not necessarily finitely generated) over a Noetherian ring.

\begin{thm}
Let $A$ be a Noetherian ring and $M$ be a non-zero $A$-module. Let $(\Ass(\widetilde{M}),<)$ be a well-ordering extension of $\Ass(\widetilde{M})$ in Example \ref{Eg: key}. Then there exists a unique filtration of $A$-modules, $(M^t)_{t\in \Ass(\widetilde{M})}$, satisfying the following properties,
\begin{enumerate}[label=\rm(\alph*)]
\item $ M^r\supset  M^s$ if $r<s$,
\item $M^0=M$, where $0$ is the minimal element of $\Ass(\widetilde M)$,
\item $\Ass(\widetilde{M^t/M^{t+1}})=\{t\}$,
\item If $t$ is a limit point of $(\Ass(\widetilde M),<)$, then
 \[ M^t=\bigcap_{r<t} \mathcal M^r.\]
\item \label{Item:Important}$\Ass(\widetilde{M^t})=\{r: r\geqslant t\}$.
\end{enumerate}

\end{thm}

We also consider coprimary filtrations of coherent sheaves on Noetherian schemes. We provide a new proof by algebraic geometry method without using Harder-Narasimhan theory. This method allows to construct the flags in a explicit way. Here is a version of our result.
\begin{thm}\label{thm: nt filtration}
Let $X$ be a Noetherian scheme, $\mathcal F$ be a non-zero coherent sheaf on $X$ and $(\Ass(\mathcal F),<)$ be a totally ordered set which is a linear extension of the partially ordered set defined in Example \ref{Eg: key}. Denote $\Ass(\mathcal F)$ by $\{\mathfrak{p}_1,\cdots, \mathfrak{p}_n\}$, where $\mathfrak{p}_1>\mathfrak{p}_2>\cdots>\mathfrak{p}_n$. Then there exists a unique chain of coherent subsheaves of $\mathcal F$,
\[0=\mathcal F_0\subset \mathcal F_1\subset \cdots \subset \mathcal F_n=\mathcal F,\]
such that $\Ass(\mathcal F_i/\mathcal F_{i-1})=\{\mathfrak{p}_i\}$, where $i\in \{1,\cdots, n\}$. Moreover,
$\mathcal F_{n-1}=\ker(\mathcal F\rightarrow i_*i^*\mathcal F)$, where $i:\Spec (\mathcal O_{X,\mathfrak{p}_n})\rightarrow X$ is a canonical morphism.
\end{thm}

The coprimary filtration is connected to the secondary representation. Kirby \cite[Theorem 3]{K1973} proved the following result: Let $A$ be a Noetherian ring, and $M$ be an $A$-module of finite length with a coprimary decomposition $M=N_1+\cdots+N_n$, where $N_i$ is $P_i$-coprimary ($i\in \{1,\cdots,n\}$). If prime ideals $P_i$ are distinct maximal ideals of $A$, then $M=N_1\oplus N_2\oplus \cdots\oplus N_n$ and this decomposition is the unique normal coprimary decomposition of $M$. It is important to mention that the definition of coprimary modules in Kirby's article differs from those usually found in many textbooks (see \cite[p.94]{E1995}). In fact, the definition of coprimary modules in this article is the same as that in \cite{E1995}. However, under the condition of Kirby's theorem mentioned above, the two definitions are equivalent (see \cite[p.44]{M1989}). Using Theorem \ref{thm: nt filtration}, we obtain the following theorem (see Theorem \ref{thm:direct sum} for details).
\begin{thm}
Let $X$ be a  Noetherian scheme and $\mathcal F$ be a non-zero coherent sheaf on $X$.
 Denote $\Ass(\mathcal F)$ by $\{\mathfrak{p}_1,\cdots, \mathfrak{p}_n\}$. Assume that $\overline{\{\mathfrak{p}_i\}}\cap \overline{\{\mathfrak{p}_j\}}=\emptyset$, where $i,j\in \{1,\cdots, n\}$ and $i\neq j$. Denote the canonical morphisms $\Spec(\mathcal O_{X,\mathfrak{p}_j})\rightarrow X$ by $t_j$, where $j\in \{1,\cdots, n\}$. Let
\[\mathcal F_i=\bigcap_{j\neq i}\ker(\mathcal F\rightarrow t_{j*}t_j^*\mathcal F).\]
  Then $\Ass(\mathcal F_i)=\{\mathfrak{p}_i\}$. Moreover,
  \[\mathcal F=\bigoplus_{1\leqslant i\leqslant n} \mathcal F_i.\]
\end{thm}

 We provide a proof of a generalization of Theorem \ref{thm: nt filtration} for coherent sheaves on locally Noetherian scheme. When operating within this framework, it is essential to utilize results from set theory. To aid in comprehension, we will review some basic concepts of set theory in section 2. As this paper concentrates on commutative algebra and algebraic geometry, we will also outline the proof for results that are not easily accessible in standard textbook, such as \cite{E1977}. Here, we state our result.
\begin{thm}
Let $X$ be a locally Noetherian scheme, $\mathcal F$ be a non-zero quasi-coherent sheaf on $X$. Let $(\Ass(\mathcal F),<)$ be a well-ordering extension of $\Ass(\mathcal F)$ in Example \ref{Eg: key}. Suppose for any affine open subscheme $U$ of $X$, $Ass(\mathcal F|_U)$ is a finite set.
Then there exists a unique filtration of quasi-coherent sheaves on $X$, $(\mathcal F^t)_{t\in \Ass(\mathcal F)}$, satisfying the following properties,
\begin{enumerate}[label=\rm(\alph*)]
\item $\mathcal F^r\supset \mathcal F^s$ if $r<s$,
\item $\mathcal F^0=\mathcal F$, where $0$ is the minimal element of $\Ass(\mathcal F)$,
\item $\Ass(\mathcal F^t/\mathcal F^{t+1})=\{t\}$,
\item If $t$ is a limit point of $(\Ass(\mathcal F),<)$, then $\bigcap\limits_{r<t} \mathcal F^r$ is a quasi-coherent sheaf and
 \[\mathcal F^t=\bigcap_{r<t} \mathcal F^r.\]
\item If $(\Ass(\mathcal F),<)$ is isomorphic to a limit ordinal, then
\[\bigcap \mathcal F^r=0.\]
      If $(\Ass(\mathcal F),<)$ is isomorphic to a successor ordinal, then $\mathcal F^t$ is coprimary where $t$ is the maximal element of $\Ass(\mathcal F)$.
\end{enumerate}

Moreover, we also have
\begin{enumerate}[label=\rm(\alph*)]
\item
\[\mathcal F^{t+1}=\ker(\mathcal F^t\rightarrow i_{*}i^*\mathcal F^t),\]
where $i:\Spec \mathcal O_{X,t}\rightarrow X$ is the canonical morphism.
\item $\Ass(\mathcal F^t)=\{r: r\geqslant t\}$.

\end{enumerate}

\end{thm}

\section{Preliminary}
\subsection{Associated points}
In this subsection, we recall and prove some basic facts of associated points of quasi-coherent sheaves on a locally Noetherian scheme.
\begin{defn}
Let $X$ be a locally Noetherian scheme, $\mathcal F$ be a quasi-coherent sheaf on $X$. We say $x$ is an associated point of $\mathcal F$ if there exists an $s_x\in \mathcal F_x$ such that $\ann(s_x)=\mathfrak{m}_x$, where $\mathfrak{m}_x$ is the maximal ideal of $\mathcal O_{X,x}$. We denote by $\Ass(\mathcal F)$ the set of all associated points of $\mathcal F$.
\end{defn}

\begin{prop}
Let $X$ be a locally Noetherian scheme, $\mathcal F$ be a quasi-coherent sheaf on $X$. Then $\Ass(\mathcal F)\subseteq \Supp(\mathcal F)$.
\end{prop}
\begin{proof}
If $x\in \Ass(\mathcal F)$, then there exists an $s_x\in \mathcal F_x$ such that $\ann(s_x)=\mathfrak{m}_x$, where $\mathfrak{m}_x$ is the maximal ideal of $\mathcal O_{X,x}$. Thus, $\mathcal F_x\neq 0$, hence $x\in \Supp(\mathcal F)$.
\end{proof}

\begin{prop}\label{prop:ann}
Let $A$ be a Noetherian ring, $M$ be an $A$-module. Then $\mathfrak{p}\in \Ass(\widetilde{M})$ if and only if there exists some $m\in M$ such that $\ann(m)=\mathfrak{p}$.
\end{prop}
\begin{proof}
See \cite[Proposition 3.1.2]{G1965}.
\end{proof}

\begin{prop}\label{prop:max}
Let $A$ be a Noetherian ring, $M$ be a non-zero $A$-module. Then every maximal element of the family of ideals, $\{\ann(m):0\neq m\in M\}$, is a prime ideal.
\end{prop}
\begin{proof}
Let $I=\ann(m)$ be a maximal element of $\{\ann(m):0\neq m\in M\}$. If $I$ is not a prime ideal, then there exists some $a,b\notin I$, such that $ab\in I$. Since $a\notin I$, we have $am\neq 0$. Consider $\ann(am)$, we have $bam=0$ since $ab\in I$, hence $b\in \ann(am)$. Therefore, $I\subset \ann(am)$, which leads to a contradiction.
\end{proof}

\begin{prop}
Let $X$ be a locally Noetherian scheme, $\mathcal F$ be a quasi-coherent sheaf on $X$. Then $\mathcal F=0$ if and only if $\Ass(\mathcal F)=\emptyset$.
\end{prop}
\begin{proof}
See \cite[Corollary 3.1.5]{G1965}.
\end{proof}

\begin{prop}
Let $X$ be a locally Noetherian scheme,
\[\xymatrix{0\ar[r]& \mathcal {F'}\ar[r]&\mathcal F\ar[r]& \mathcal {F"}\ar[r]& 0}\]
be a short exact sequence of quasi-coherent sheaves on $X$. Then $\Ass(\mathcal F')\subseteq \Ass(\mathcal F)\subseteq \Ass(\mathcal F')\cup \Ass(\mathcal F")$.
\end{prop}
\begin{proof}
See \cite[Proposition 3.1.7]{G1965}.
\end{proof}

\begin{prop}
Let $X$ be a locally Noetherian scheme, $\mathcal F$ be a quasi-coherent sheaf on $X$ and $x\in X$. Then $x\in \Ass(\mathcal F)$ if and only if there exists an open neighborhood $U$ of $x$ and a section $s\in \Gamma(U,\mathcal F)$, such that $x$ is a generic point of $\Supp(\mathcal O_Us)$.
\end{prop}
\begin{proof}
See \cite[Proposition 3.1.3]{G1965}.
\end{proof}

\begin{prop}\label{prop:pull-back}
Let $X$ be a locally Noetherian scheme, $\mathcal F$ be a quasi-coherent sheaf on $X$, $x\in X$ and $i:\Spec(\mathcal O_{X,x})\rightarrow X$ be the canonical morphism. Then $\mathfrak{p}\in \Ass(i^*\mathcal F)$ if and only if $i(\mathfrak{p})\in \Ass(\mathcal F)$.
\end{prop}
\begin{proof}
Without loss of generality, we may assume $X=\Spec(A)$, where $A$ is a Noetherian ring. Let $\mathcal F=\widetilde M$, where $M$ is an $A$-module. Let $\mathfrak{p}\in \Spec(\mathcal O_{X,x})$, we have $M_{i(p)}=(M_x)_\mathfrak{p}$, hence $(i^*\mathcal F)_\mathfrak{p}=\mathcal F_{i(\mathfrak{p})}$. Therefore,
$\mathfrak{p}\in \Ass(i^*\mathcal F)$ if and only if $i(\mathfrak{p})\in \Ass(\mathcal F)$.
\end{proof}

\begin{prop}\label{prop:min}
Let $X$ be a locally Noetherian scheme, $\mathcal F$ be a quasi-coherent sheaf on $X$,
$\mathfrak{p}$ be a generic point of $\Supp(\mathcal F)$. Then $\mathfrak{p}\in \Ass(\mathcal F)$.
\end{prop}
\begin{proof}
Let $U=\Spec(A)$ be an open neighborhood of $\mathfrak{p}$, where $A$ is a Noetherian ring and $\mathcal F|_U=\widetilde{M}$, where $M$ is an $A$-module. Let $i: \Spec(A_\mathfrak{p})\rightarrow X$ be the canonical morphism. We identify elements of $\Spec(A_\mathfrak{p})$ with $\{\mathfrak{q}\in \Spec(A):\mathfrak{q}\subseteq \mathfrak{p}\}$. Let $\mathfrak{q}\in \Spec(A)$ such that $\mathfrak{q}\subset \mathfrak{p}$. Since $\mathfrak{p}$ is a generic point of $\Supp(\mathcal F)$, we have $M_\mathfrak{q}=0$, and thus $(i^*\mathcal F)_\mathfrak{q}=0$. Note that \[\Ass(i^*\mathcal F)\subseteq \Supp(i^*\mathcal F)=\{\mathfrak{p}\},\]
 we have $\Ass(i^*\mathcal F)=\{\mathfrak{p}\}$.  By Proposition \ref{prop:pull-back}, $\mathfrak{p}\in \Ass(\mathcal F)$.
\end{proof}

\begin{prop}\label{prop:coprimary}
Let $A$ be a Noetherian ring, $M$ be an $A$-module, and $\mathfrak{p}\in \Spec(A)$. Let $r_a: M\rightarrow M$ be the $A$-module homomorphism defined as follows: for any $a \in A$ and $m\in M$, $r_a(m)=am$. Then $\Ass(\widetilde M)=\{\mathfrak{p}\}$ if and only if the following conditions hold.
\begin{enumerate}
\item If $a\in \mathfrak{p}$, then for any $m\in M$, there exists some $n\in \mathbb{N}$, such that $a^nm=0$.
\item If $a\notin \mathfrak{p}$, then $r_a$ is injective.
\end{enumerate}
\end{prop}
\begin{proof}
$\Rightarrow$ Let $0\neq m\in M$ and $N=Am$. Observe that $\Ass(\widetilde{N})\subseteq \Ass(\widetilde{M})=\{\mathfrak{p}\}$, we have $\Ass(\widetilde{N})=\{\mathfrak{p}\}$. By Proposition \ref{prop:max}, we have $\ann(m)\subseteq \mathfrak{p}$. If $a\notin \mathfrak{p}$, then $a\notin \ann(m)$, hence $am\neq 0$. Note that for any $m\neq 0$, $am\neq 0$, we can see that $r_a$ is injective. Denote $\ann(m)$ by $I$.
 Assume that $a\in \mathfrak{p}$. Notice that $N\cong A/I$, we have $\Supp(\widetilde{N})=\Supp(\widetilde{A/I})=\{\mathfrak{q}: \mathfrak{q}\supseteq I, \text{$\mathfrak{q}$ is a prime}\}$. Assume that $\mathfrak{q}\in \Supp(\widetilde{N})$, by Proposition \ref{prop:min}, we have $\mathfrak{q}\supseteq \mathfrak{p}$ since $\Ass(\widetilde{N})=\{\mathfrak{p}\}$. Therefore, $\sqrt{I}=\mathfrak{p}$. Thus there exists some $n$ such that $a^n\in I$, hence $a^nm=0$. \\

$\Leftarrow$ Assume that $m$ is a non-zero element of $M$ such that
$\ann(m)=\mathfrak{q}$, where $\mathfrak{q}$ is a prime. Assume $a\in \mathfrak{q}\backslash \mathfrak{p}$. Since $r_a$ is injective, $am\neq 0$. However, $a\in \mathfrak{q}=\ann(m)$, so $am=0$, which leads to a contradiction. We have
$\mathfrak{q}\subseteq \mathfrak{p}$. If $a\in \mathfrak{p}$, then there exists some $n$ such that $a^nm=0$, so $a^n\in \mathfrak{q}$. Since $\mathfrak{q}$ is a prime, $a\in \mathfrak{q}$, hence $\mathfrak{p}\subseteq \mathfrak{q}$. Therefore, $\mathfrak{p}=\mathfrak{q}$, thus $\Ass(\widetilde M)=\{\mathfrak{p}\}$.
\end{proof}

\begin{prop}\label{prop: ass localization}
Let $A$ be a Noetherian ring, $M$ be an $A$-module, $\mathfrak{p}\in \Spec(A)$ and
$i:\Spec(A_\mathfrak{p})\rightarrow \Spec(A)$ be a canonical morphism. Let $\mathcal F=\widetilde{M_\mathfrak{p}}$ be a quasi-coherent sheaf on $\Spec(A_\mathfrak{p})$. If $\Ass(\mathcal F)=\{\mathfrak{p}_\mathfrak{p}\}$, then $\Ass(i_*\mathcal F)=\{\mathfrak{p}\}$.
\end{prop}
\begin{proof}
Let $r_a$ be an $A$-module homomorphism such that $r_a(m)=am$, where $a\in A$ and $m\in M_\mathfrak{p}$.
Since $\Ass(\mathcal F)=\{\mathfrak{p}_\mathfrak{p}\}$, by Proposition \ref{prop:coprimary}, $r_a$ is injective if $a\notin \mathfrak{p}$. If $a\in \mathfrak{p}$, then for any $m\in M_\mathfrak{p}$, there exists some $n$ such that $a^nm=0$. Notice that $i_*\mathcal F\cong \widetilde{M_\mathfrak{p}}$, by Proposition \ref{prop:coprimary},
$\Ass(i_*\mathcal F)=\{\mathfrak{p}\}$.
\end{proof}

\begin{prop}\label{prop: open embedding}
Let $X$ be a locally Noetherian scheme, $U$ be an open subset of $X$, $j:U\rightarrow X$ be the canonical open embedding, and $\mathcal F$ be a quasi-coherent sheaf on $U$. Then $\Ass(j_*\mathcal F)=\Ass(\mathcal F)$.
\end{prop}
\begin{proof}
See \cite[Proposition 3.1.13]{G1965}.
\end{proof}
\subsection{Reminder on partially ordered sets}
\begin{defn}
Let $(S,<)$ be a partially ordered set. The relation $<$ is called well-founded if
for any non-empty subset $T\subseteq S$, there exists some $x\in T$ such that there does not exist some $y\in T$ such that $y<x$.
In this case, we say $x$ is a $<$-minimal element of $T$.
\end{defn}

The following Proposition \ref{prop: well-founded} is actually a special case of an exercise in \cite[p. 230]{D2014}.
\begin{prop}\label{prop: well-founded}
Let $(S,<)$ be a partially ordered set. Then $<$ is a well-founded relation if and only
if there are no infinite descending sequences.
\end{prop}
\begin{proof}
$\Rightarrow$ Assume $\{x_n\}$ be a descending sequence of $S$. Consider the non-empty subset $T=\{x_1,x_2,\cdots,x_n,\cdots\}$ of $S$, there exists some $x_n$ such that $x_n$ is a $<$-minimal element of $T$. Therefore, $x_{n+1}\notin T$, which means $T$ has only $n$ elements. \\

$\Leftarrow$
Let $T$ be a non-empty subset of $S$. We assume $T$ does not have a $<$-minimal element. Let $x_1\in T$. Since $x_1$ is not a $<$-minimal element in $T$, there exists $x_2\in T$ such that $x_2<x_1$. Repeat this process. We will find an infinite sequence $x_1>x_2>\cdots >x_n>\cdots$,
which leads to a contradiction.
\end{proof}

\begin{eg}\label{Eg: key}
Let $X$ be a locally Noetherian scheme. We equip $X$ with a partial order. Let $\mathfrak{p},\mathfrak{q}\in X$, we say $\mathfrak{q}<\mathfrak{p}$ if $\overline{\{\mathfrak{p}\}}\subset \overline{\{\mathfrak{q}\}}$. Let $\mathcal F$ be a non-zero quasi-coherent sheaf on $X$, $(\Ass(\mathcal F),<)$ be a partially ordered set, where the partially order $<$ is the restriction of the partially order of $X$ on $\Ass(\mathcal{F})$. Let $\mathfrak{p}_1>\mathfrak{p}_2>\cdots>\mathfrak{p}_n>\cdots$ be a descending sequence of $\Ass(\mathcal F)$, and $\mathfrak{p}_1\in \Spec(A)$, where $A$ is a Noetherian ring. Note that for any $n$, $\mathfrak{p}_1\in \overline{\{\mathfrak{p}_n\}}$, we have $\mathfrak{p}_n\in \Spec (A)$. Since $A$ is a Noetherian ring, $\mathfrak{p}_1$ is a finitely generated ideal of $A$ which is generated by $m$ elements.  By Krull's height theorem \cite[Theorem 10.2]{E1995},  the sequence $x_1>x_2>\cdots>x_n>\cdots$ has at most $m+1$ elememts. By Proposition \ref{prop: well-founded}, $<$ is a well-founded relation on $\Ass(\mathcal F)$.
\end{eg}

Since $<$ is well-founded on $\Ass(\mathcal F)$, we could define rank function (\cite[p.68]{J2003}),
\[\rk:\Ass(\mathcal F)\rightarrow \Ord.\] It is defined as follows,
\[\rk(x)=\sup\{\rk(y)+1:y<x\}.\]
By definition, we can see that if $x$ is a $<$-minimal element of $\Ass(\mathcal{F})$, then $\rk(x)=0$. Moreover,
if $x<y$, then $\rk(x)<\rk(y)$.\\

The following Proposition is a variation of Szpilrajn's Theorem. We only give an outline of the proof and recommend the readers to refer \cite{BP1982} for more details.

\begin{prop}
Any well-founded partially ordered set can be extend to a well-ordered set.
\end{prop}
\begin{proof}
Let $(S,<)$ be a well-founded partially ordered set. Let
 \[S_\alpha=\{x\in S: \rk(x)=\alpha\},\]
where $\alpha\in \Ord$. By the well-ordering theorem, there exists a well-order $<_\alpha$ on $S_\alpha$. Let $x,y\in S$, we say $x<' y$ if one of the following holds,
\begin{enumerate}
\item $\rk (x)<\rk (y)$,
\item  there exists some $\alpha\in \Ord$, such that $x<_\alpha y$.
\end{enumerate}
Therefore, $(S,<')$ is a linear extension of $(S,<)$ and it is well-ordered.
\end{proof}

\begin{defn}
Let $(S, <)$ be a well-ordered set. The element $s$ of $S$ is called a limit point of $S$
if there is some element $x$ of $S$ such that $x<s$, and for every element $x\in S$ with $x<s$, there is some $y\in S$ such that both $x<y$ and $y<s$ hold. Equivalently, $s$ is a limit point of $S$ under the order topology.
\end{defn}

\section{Coprimary filtrations}
\begin{lemma}\label{lemma:calculation}
Let $X$ be a locally Noetherian scheme, $\mathcal F$ be a non-zero quasi-coherent sheaf on $X$, $\mathfrak{p}\in \Ass(\mathcal F)$ and $\rk(\mathfrak{p})=0$. Let $i:\Spec(\mathcal O_{X,\mathfrak{p}})\rightarrow X$ be the canonical morphism. Then $\Ass(i_{*}i^{*}\mathcal F)=\{\mathfrak{p}\}$.
\end{lemma}
\begin{proof}
Let $U$ be an affine open neighborhood of $\mathfrak{p}$. We denote $U$ by $\Spec(A)$, where $A$ is a Noetherian ring. Let  $j: U\rightarrow X$ be the canonical open embedding. $\mathcal F|_U\cong \widetilde{M}$, where $M$ is an $A$-module. We can see that $i^*\mathcal F\cong \widetilde{M_\mathfrak{p}}$. Let $\mathfrak{q}_\mathfrak{p}\in \Spec(A_\mathfrak{p})$, we claim that $M_\mathfrak{q}=0$ if $\mathfrak{q}\neq \mathfrak{p}$. Otherwise, $\mathfrak{q}\in \Supp(\widetilde{M})$. If such $\mathfrak{q}$ exists, we may assume $\mathfrak{q}$ be a $<$-minimal element in $\Supp(\widetilde{M})$. By Proposition \ref{prop:min}, $\mathfrak{q}\in \Ass(\widetilde{M})$, thus $\mathfrak{q}\in \Ass(\mathcal F)$, which contradicts to $\rk(\mathfrak{p})=0$. Therefore, $\mathfrak{q}_\mathfrak{p}\notin \Supp(i^*\mathcal F)$ if $\mathfrak{q}\neq \mathfrak{p}$. Thus, $\Ass(i^*\mathcal F)=\{\mathfrak{p}_\mathfrak{p}\}$ since $M_\mathfrak{p}\neq 0$. By Proposition \ref{prop: ass localization}, $\Ass_U(\widetilde{M_\mathfrak{p}})=\{\mathfrak{p}\}$. Observe that $j_*\widetilde{M_\mathfrak{p}}\cong i_*i^*\mathcal F$, by Proposition \ref{prop: open embedding}, $\Ass(i_{*}i^{*}\mathcal F)=\{\mathfrak{p}\}$.
\end{proof}

\begin{lemma}\label{Lemma:key}
Let $X$ be a locally Noetherian scheme, $\mathcal F$ be a non-zero quasi-coherent sheaf on $X$, $\mathfrak{p}\in \Ass(\mathcal F)$ and $\rk(\mathfrak{p})=0$. Let $i:\Spec(\mathcal O_{X,\mathfrak{p}})\rightarrow X$ be the canonical morphism, $\mathcal G$ be a quasi-coherent subsheaf of $\mathcal F$. Then $\mathcal G=\ker(\mathcal F\rightarrow i_{*}i^* \mathcal F)$ if and only if both $\Ass(\mathcal F/\mathcal G)=\{\mathfrak{p}\}$ and
$\Ass(\mathcal G)=\Ass(\mathcal F)\backslash\{\mathfrak{p}\}$ hold.
\end{lemma}
\begin{proof}
Let $U$ be an affine open neighborhood of $\mathfrak{p}$. We denote $U$ by $\Spec(A)$, where $A$ is a Noetherian ring. Let $j: U\rightarrow X$ be the canonical open embedding. $\mathcal F|_U\cong \widetilde{M}$, where $M$ is an $A$-module. We can see that $i_{*}i^* \mathcal F|_U\cong\widetilde{M_\mathfrak{p}}$. Hence, $(i_{*}i^* \mathcal F)_\mathfrak{p}\cong \mathcal F_\mathfrak{p}$.
Since the sequence
\[\xymatrix{0\ar[r]& \mathcal G\ar[r]&\mathcal F\ar[r]& i_{*}i^* \mathcal F}\]
is exact, we have an exact sequence,
\[\xymatrix{0\ar[r]& \mathcal G_\mathfrak{p}\ar[r]&\mathcal F_\mathfrak{p}\ar[r]& \mathcal F_\mathfrak{p},}\]
when we take the stalk at $\mathfrak{p}$.
Thus, $\mathcal G_\mathfrak{p}=0$, which means $\mathfrak{p}\notin \Supp(\mathcal G)$, so $\mathfrak{p}\notin \Ass(\mathcal G)$.
We can also see that $\mathcal F/\mathcal G\neq 0$.
By Lemma \ref{lemma:calculation}, $\Ass(i_{*}i^* \mathcal F)=\{\mathfrak{p}\}$. Because both $\Ass(\mathcal F/\mathcal G)\subseteq \Ass(i_{*}i^* \mathcal F)$ and  $\mathcal F/\mathcal G\neq 0$ hold,
$\Ass(\mathcal F/\mathcal G)=\{\mathfrak{p}\}$.
Notice that $\Ass(\mathcal F)\subseteq \Ass(\mathcal G)\cup \Ass(\mathcal F/\mathcal G)$, $\Ass( \mathcal F)\backslash\{\mathfrak{p}\}\subseteq \Ass(\mathcal G)$. Since both $\mathfrak{p}\notin \Ass(\mathcal G)$ and
 $\Ass(\mathcal G)\subseteq \Ass(\mathcal F)$ hold,
\[\Ass(\mathcal F)\backslash\{\mathfrak{p}\}=\Ass(\mathcal G).\]

Now, we prove the converse. Assume $\mathcal G_\mathfrak{p}\neq 0$, so $\mathfrak{p}\in \Supp(\mathcal G)$. Hence there exists some $\mathfrak{q}\in \Ass(\mathcal G)$ such that $\mathfrak{p}\in \overline{\{\mathfrak{q}\}}$. However, since $\mathfrak{p}$ is minimal in $\Ass(\mathcal F)$, for any $\mathfrak{q}\in \Ass(\mathcal F)$ such that $\mathfrak{q}\neq \mathfrak{p}$, $\mathfrak{p}\notin \overline{\{\mathfrak{q}\}}$. This leads to a contradiction because $\Ass(\mathcal G)=\Ass(\mathcal F)\backslash\{\mathfrak{p}\}$. Note that $i_*i^*\mathcal G=i_*\widetilde{\mathcal G_\mathfrak{p}}=0$
since $\mathcal G_\mathfrak{p}=0$. Consider the following commutative diagram.
\[\xymatrix{\mathcal G\ar[r]^g\ar[d]&\mathcal F\ar[d]^h\\
i_*i^*\mathcal G\ar[r]&i_*i^*\mathcal F}\]
Since $i_*i^*\mathcal G=0$, the morphism $h\circ g: \mathcal G\rightarrow i_*i^*\mathcal F$ is zero. Let $\mathcal G'=\ker (\mathcal F\rightarrow i_*i^*\mathcal F)$. We have $\mathcal G\subseteq \mathcal G'$ because $h\circ g=0$. Consider the short exact sequence $0\rightarrow \mathcal G'/\mathcal G\rightarrow \mathcal F/\mathcal G\rightarrow \mathcal F/\mathcal G'\rightarrow 0$. We can see that
$\Ass(\mathcal G'/\mathcal G)\subseteq \Ass(\mathcal F/\mathcal G)=\{\mathfrak{p}\}$. Assume $\mathcal G\neq \mathcal G'$, thus $\Ass(\mathcal G'/\mathcal G)=\{\mathfrak{p}\}$, so
\[\mathfrak{p}\in \Ass(\mathcal G'/\mathcal G)\subseteq \Supp(\mathcal G'/\mathcal G)\subseteq \Supp(\mathcal G').\]
Note that  the definition of $\mathcal G'$ here is identical to the definition of $\mathcal G$ in the `only if' part of the proof, we have $\mathcal G'_\mathfrak{p}=0$, which leads to a contradiction. Therefore, $\mathcal G=\mathcal G'$.

\end{proof}

\begin{defn}
Let $X$ be a scheme, $\mathcal F$ be an $\mathcal{O}_X$-module and $\{\mathcal F_i, i\in I\}$ be a family of $\mathcal{O}_X$-submodules of $\mathcal F$. The intersection $\bigcap\limits_{i\in I} \mathcal{F}_i$ of the family $\{\mathcal{F}_i\}$ is the $\mathcal{O}_X$-submodule of $\mathcal F$ defined as the kernel
of the canonical homomorphism
\[\mathcal F \rightarrow \prod\limits_{i\in I}(\mathcal F/\mathcal F_i).\]

\end{defn}

\begin{thm}\label{Thm: coprimary filtration}
Let $X$ be a locally Noetherian scheme, $\mathcal F$ be a non-zero quasi-coherent sheaf on $X$. Let $(\Ass(\mathcal F),<)$ be a well-ordering extension of $\Ass(\mathcal F)$ in Example \ref{Eg: key}. Suppose one of the following holds:
\begin{enumerate}
\item \label{1}$(\Ass(\mathcal F),<)$ is isomorphic to the least nonzero limit ordinal $\omega$, and for any $p\in X$, there exists only finitely many $\mathfrak{q}\in \Ass(\mathcal F)$ such that $\mathfrak{p}\in\overline{\{\mathfrak{q}\}}$.
\item \label{2}For any affine open subscheme $U$ of $X$, $Ass(\mathcal F|_U)$ is a finite set.
\end{enumerate}

Then there exists a filtration of quasi-coherent sheaves on $X$, $(\mathcal F^t)_{t\in \Ass(\mathcal F)}$, satisfying the following properties,
\begin{enumerate}[label=\rm(\alph*)]
\item $\mathcal F^r\supset \mathcal F^s$ if $r<s$,
\item $\mathcal F^0=\mathcal F$, where $0$ is the minimal element of $\Ass(\mathcal F)$,
\item $\Ass(\mathcal F^t/\mathcal F^{t+1})=\{t\}$,
\item If $t$ is a limit point of $(\Ass(\mathcal F),<)$, then $\bigcap\limits_{r<t} \mathcal F^r$ is a quasi-coherent sheaf and
 \[\mathcal F^t=\bigcap_{r<t} \mathcal F^r.\]
\item If $(\Ass(\mathcal F),<)$ is isomorphic to a limit ordinal, then
\[\bigcap \mathcal F^r=0.\]
      If $(\Ass(\mathcal F),<)$ is isomorphic to a successor ordinal, then $\mathcal F^t$ is coprimary where $t$ is the maximal element of $\Ass(\mathcal F)$.
\end{enumerate}
If assumption \ref{2} holds, such filtration is unique.
Moreover, we also have
\begin{enumerate}[label=\rm(\alph*)]
\item
\[\mathcal F^{t+1}=\ker(\mathcal F^t\rightarrow i_{*}i^*\mathcal F^t),\]
where $i:\Spec \mathcal O_{X,t}\rightarrow X$ is the canonical morphism.
\item $\Ass(\mathcal F^t)=\{r: r\geqslant t\}$.

\end{enumerate}
\end{thm}
\begin{proof}
Step 1: Construction of $\mathcal F^t$.\\
We define $\mathcal F^t$ by transfinite recursion,
\begin{enumerate}[label=\rm(\alph*)]
\item \label{a}$\mathcal F^0=\mathcal F$,
\item \label{b} If $t$ is a successor of an element $s$ (i.e. $t=s+1$), then
$\mathcal F^{t}=\ker(\mathcal F^s\rightarrow i_{*}i^*\mathcal F^s)$,
 \item  \label{c}If $t$ is a limit point, then\begin{equation*}
 \mathcal F^t=\bigcap_{r<t} \mathcal F^r.\end{equation*}
\end{enumerate}

Step 2: Quasi-Coherence of $\mathcal F^t$.\\
We claim that for any $t\in \Ass(\mathcal F)$, $\mathcal F^t$ is quasi-coherent and $\Ass(\mathcal F^t)=\{r: r\geqslant t\}$. We prove the claim by transfinite induction.
When $t=0$, the claim naturally holds. Now, we assume for any $s$ such that $s<t$, $\mathcal F^s$ is quasi-coherent and $\Ass(\mathcal F^s)=\{r: r\geqslant s\}$. We want to prove $\mathcal F^{t}$ is quasi-coherent and $\Ass(\mathcal F^{t})=\{r: r\geqslant t\}$. First, we consider the case that $t$ is a successor of some element $s$. Since $i: \Spec \mathcal O_{X,s}\rightarrow X$ is quasi-compact and quasi-separated, $ i_{*}i^{*}\mathcal F^s$ is quasi-coherent, and so is $\mathcal F^{t}$. Since $s$ is the minimal element in $\Ass(\mathcal F^s)$, by Lemma \ref{Lemma:key}, $\Ass(\mathcal F^{t})=\{r: r\geqslant t\}$ and $\Ass(\mathcal F^s/\mathcal F^{t})=\{s\}$. Secondly, we assume that $t$ is a limit point. If such $t$ exists, $(\Ass(\mathcal F),<)$ is not isomorphic to $\omega$. We may assume that for any affine open subscheme $U$ of $X$, $Ass(\mathcal F|_U)$ is a finite set. Let $U=\Spec(A)\subseteq X$, where $A$ is a Noetherian ring, $\Ass(\mathcal F^s|_U)$ is a finite set when $s<t$. Since $t$ is a limit point, there exists an $r<t$ such that for any $r'$ satisfying $r\leqslant r'<t$, $\Ass(\mathcal F^r|_U)=\Ass(\mathcal F^{r'}|_U)$.
Since $r\notin \Ass(\mathcal F^{r+1})$, $r\notin \Ass(\mathcal F^{r+1}|_U)$.
Note that $\Ass(\mathcal F^r|_U)=\Ass(\mathcal F^{r+1}|_U)$, so $r\notin \Ass(\mathcal F^r|_U)$, hence $r\notin U$.
Consider the following short exact sequence.
\[\xymatrix{0\ar[r]& \mathcal F^r|_U\ar[r]&\mathcal F^{r+1}|_U\ar[r]& (\mathcal F^r/\mathcal F^{r+1})|_U\ar[r]& 0}\]
Note that $\Ass((\mathcal F^r/\mathcal F^{r+1})|_U)=\Ass(\mathcal F^r/\mathcal F^{r+1})\cap U=\emptyset$, hence $(\mathcal F^r/\mathcal F^{r+1})|_U=0$. Therefore, $\mathcal F^r|_U=\mathcal F^{r+1}|_U$. By transfinite induction, we can see that $\mathcal F^{r'}|_U=\mathcal F^r|_U$ for any $r'$ satisfying $r\leqslant r'<t$, hence $\mathcal F^t|_U=\mathcal F^r|_U$. Therefore, $\mathcal F^t$ is quasi-coherent. Note that for any $r<t$, $\Ass(\mathcal F^{t})\subseteq \Ass(\mathcal F^r)=\{s: s\geqslant r\}$, so $\Ass(\mathcal F^t)\subseteq \{s: s\geqslant t\}$. For any $s$ such that $s\geqslant t$, there exists an affine open subset $U=\Spec(A)$ of $X$ such that $s\in U$. Note that there exists an $r$ satisfying $r<t$ such that $\mathcal F^r|_U=\mathcal F^t|_U$. Since $s\in \Ass(\mathcal F^r)\cap U$, $s\in \Ass(\mathcal F^t)$. Therefore,
$\Ass(\mathcal F^{t})= \{s: s\geqslant t\}$. \\

Step 3: Existence of Coprimary Filtration.\\
To prove the existence, we only need to prove if $(\Ass(\mathcal F),<)$ is isomorphic to a limit ordinal, $\bigcap \mathcal F^r=0$. We first prove it under the assumption \ref{1}. Otherwise, there exists some $\mathfrak{p}\in X$ such that $(\bigcap \mathcal F^r)_\mathfrak{p}\neq 0$. Thus, $(\mathcal F^t)_\mathfrak{p}\neq 0$ for any $t\in \Ass(\mathcal F)$. By Proposition \ref{prop:min}, there exists some $s\in\Ass(\mathcal F^t)$ such that $\mathfrak{p}\in \overline{\{s\}}$. Since there are only finitely many $s$ satisfying $\mathfrak{p}\in \overline{\{s\}}$, this leads to a contradiction. Now we prove it under the assumption \ref{2}. Note that $\bigcap \mathcal F^r$ is quasi-coherent,
$\Ass(\bigcap \mathcal F^r)\subseteq \Ass(\mathcal F^t)=\{s: s\geqslant t\}$ for any $t\in \Ass(\mathcal F)$. Thus, $\Ass(\bigcap \mathcal F^r)=\emptyset$, so $\bigcap \mathcal F^r=0$. \\

Step 4: Uniqueness of Coprimary Filtration.\\
 We assume $\mathcal G^t$ also satisfies the properties. We prove $\mathcal G^t=\mathcal F^t$ by transfinite induction. If the set $S=\{r\in \Ass(\mathcal F): \mathcal G^r\neq\mathcal F^r\}$ is non-empty, we may take $t$ to be the minimal element of $S$. Then $t\neq 0$ and $t$ is not a limit point. Now we may assume $t=r+1$. By the definition of $t$, we have $\mathcal F^r=\mathcal G^r$, hence $\Ass(\mathcal G^r)=\{s:s\geqslant r\}$. Since both $\Ass(\mathcal G^r)\subseteq \Ass(\mathcal G^{r+1})\cup\Ass(\mathcal G^r/\mathcal G^{r+1})$ and $\Ass(\mathcal G^r/\mathcal G^{r+1})=\{r\}$ hold,
\[\Ass(\mathcal G^t)\supseteq \{s:s\geqslant t\}.\]
 Let $U=\Spec (A)$ be an affine open neighborhood of $r$. $\Ass(\mathcal G^r|_U)$ is a finite set. Denote $\Ass(\mathcal G^r|_U)$ by $\{s_1,s_2,\cdots,s_n\}$, where $s_1<s_2<\cdots<s_n$ and $s_1=r$. We consider the following filtration of $\mathcal G^r|_U$.
\[0\subseteq \mathcal G^{s_{n-1}+1}|_U\subseteq \cdots \subseteq\mathcal G^{s_2+1}|_U\subseteq\mathcal G^{s_1+1}|_U\subseteq \mathcal G^r|_U\]
We claim that $\mathcal G^{s_n+1}|_U=0$ if $s_n$ is not the maximal element of $\Ass(\mathcal F)$.
We first claim that $\mathcal G^s|_U=\mathcal G^{s_n+1}|_U$ if $s\geqslant s_n+1$. Assume the claim is not true, there exists a minimal $s'$ such that  $\mathcal G^{s'}|_U\neq\mathcal G^{s_n+1}|_U$. If $s'$ is a limit point,
\[\mathcal G^{s'}|_U=\bigcap_{r<s'} (\mathcal G^r|_U)=\mathcal G^{s_n+1}|_U,\]
which leads to a contradiction. Now we may assume that $s'=s"+1$, consider the following short exact sequence.
\[\xymatrix{0\ar[r]& \mathcal G^{s"+1}|_U\ar[r]&\mathcal G^{s"}|_U\ar[r]& (\mathcal G^{s"}/\mathcal G^{s"+1})|_U\ar[r]& 0}\]
Since $s"\geqslant s_n+1>s_n$, $\Ass((\mathcal G^{s"}/\mathcal G^{s"+1})|_U)=\{s"\}\cap U=\{s"\}\cap \Ass(\mathcal G^r|_U)=\emptyset$, therefore $(\mathcal G^{s"}/\mathcal G^{s"+1})|_U=0$, and consequently $\mathcal G^{s"}|_U=\mathcal G^{s'}|_U$, which leads to a contradiction. If $(\Ass(\mathcal F),<)$ is isomorphic to a limit ordinal, then $0=\bigcap (\mathcal G^{s}|_U)= \mathcal G^{s_n+1}|_U$. If $(\Ass(\mathcal F),<)$ is isomorphic to a successor ordinal, then $\mathcal G^{s_{n}+1}|_U=\mathcal G^s|_U$ where $s$ is the maximal element of $\Ass(\mathcal F)$. Note that $\Ass(\mathcal G^s|_U)=\{s\}\cap U=\emptyset$,  we have $\mathcal G^s|_U=0$. Therefore, $\mathcal G^{s_n+1}|_U=0$.
Now, we claim that $\mathcal G^s|_U=\mathcal G^{s_i+1}|_U$ if $s_{i+1}\geqslant s\geqslant s_i+1$, where $1\leqslant i\leqslant n-1$. Assume the claim is not true, there exists a minimal $s'$ satisfying $s_{i+1}\geqslant s'> s_i+1$ such that  $\mathcal G^{s'}|_U\neq\mathcal G^{s_i+1}|_U$. If $s'$ is a limit point,
\[\mathcal G^{s'}|_U=\bigcap_{r<s'} \mathcal G^r|_U=\mathcal G^{s_i+1}|_U,\]
which leads to a contradiction. Now we may assume that $s'=s"+1$, consider the following short exact sequence.
\[\xymatrix{0\ar[r]& \mathcal G^{s"+1}|_U\ar[r]&\mathcal G^{s"}|_U\ar[r]& (\mathcal G^{s"}/\mathcal G^{s"+1})|_U\ar[r]& 0}\]
Since $s_{i+1}\geqslant s'>s"\geqslant s_i+1>s_i$, $\Ass((\mathcal G^{s"}/\mathcal G^{s"+1})|_U)=\{s"\}\cap U=\{s"\}\cap \Ass(\mathcal G^r|_U)=\emptyset$, therefore $(\mathcal G^{s"}/\mathcal G^{s"+1})|_U=0$, and consequently $\mathcal G^{s"}|_U=\mathcal G^{s'}|_U$, which leads to a contradiction. From the discussion above, we can see that
\begin{enumerate}
\item $\Ass(\mathcal G^{s_{n-1}+1}|_U)=\{s_n\}$,
\item $\Ass((\mathcal G^{s_{i}+1}/\mathcal G^{s_{i+1}+1})|_U)=\{s_{i+1}\}$, where $1\leqslant i\leqslant n-2$.
\end{enumerate}
Note that $t=s_1+1$, $\Ass(\mathcal G^t|_U)\subseteq \{s_2,\cdots, s_n\}$, hence $r=s_1\notin \Ass(\mathcal G^t)$. Therefore, $\Ass(\mathcal G^t)=\Ass(\mathcal G^r)\backslash\{r\}$. By Lemma \ref{Lemma:key},
we can see that $\mathcal F^t=\mathcal G^t$.
\end{proof}

\begin{defn}
Let $X$ be a locally Noetherian scheme, $\mathcal F$ be a non-zero coherent sheaf on $X$ and
 $\Ass (\mathcal F)$ be a well-ordering extension of $\Ass(\mathcal F)$ in Example \ref{Eg: key}.
 The unque filtration of coherent sheaves on $X$, $(\mathcal F^t)_{t\in \Ass(\mathcal F)}$, defined in Theorem
 \ref{Thm: coprimary filtration} is called a coprimary filtration of $\mathcal F$. In particular, if $X$ is a Noetherian scheme, then a coprimary filtration of $\mathcal F$ could be written in the form:
 \[0=\mathcal F^0\subseteq \mathcal F^1\subseteq \cdots \subseteq  \mathcal F^{n-1}\subseteq \mathcal F^n=\mathcal F.\]
  .
\end{defn}

\begin{thm}\label{Thm:commutative ring}
Let $A$ be a Noetherian ring and $M$ be a non-zero $A$-module. Let $(\Ass(\widetilde{M}),<)$ be a well-ordering extension of $\Ass(\widetilde{M})$ in Example \ref{Eg: key}. Then there exists a unique filtration of $A$-modules, $(M^t)_{t\in \Ass(\widetilde{M})}$, satisfying the following properties,
\begin{enumerate}[label=\rm(\alph*)]
\item $ M^r\supset  M^s$ if $r<s$,
\item $M^0=M$, where $0$ is the minimal element of $\Ass(\widetilde M)$,
\item $\Ass(\widetilde{M^t/M^{t+1}})=\{t\}$,
\item If $t$ is a limit point of $(\Ass(\widetilde M),<)$, then
 \[ M^t=\bigcap_{r<t} \mathcal M^r.\]
\item \label{Item:Important}$\Ass(\widetilde{M^t})=\{r: r\geqslant t\}$.
\end{enumerate}
Moreover, we also have
\begin{enumerate}[label=\rm(\alph*)]
\item
\[M^{t+1}=\ker( M^t\rightarrow  (M^t)_t),\]
\item
If $(\Ass(\widetilde M),<)$ is isomorphic to a limit ordinal, then
\[\bigcap  M^r=0.\]
      If $(\Ass(\widetilde M),<)$ is isomorphic to a successor ordinal, then $ M^t$ is coprimary where $t$ is the maximal element of $\Ass(\widetilde M)$.

\end{enumerate}
\end{thm}
\begin{proof}
Step 1: Construction of $M^t$.\\
We define $M^t$ by transfinite recursion,
\begin{enumerate}[label=\rm(\alph*)]
\item \label{a}$M^0=M$,
\item \label{b} If $t$ is a successor of an element $s$ (i.e. $t=s+1$), then
$M^{t}=\ker(M^s\rightarrow (M^s)_s)$,
 \item  \label{c}If $t$ is a limit point, then\begin{equation*}
  M^t=\bigcap_{r<t} M^r.\end{equation*}
\end{enumerate}

Step 2: Associated Points of $\widetilde{M^t}$.\\
We claim that for any $t\in \Ass(\widetilde M)$, $\Ass(\widetilde{M^t})=\{r: r\geqslant t\}$. We prove the claim by transfinite induction.
When $t=0$, the claim naturally holds. Now, we assume for any $s$ such that $s<t$, $\Ass(\widetilde{ M^s})=\{r: r\geqslant s\}$. We want to prove $\Ass(\widetilde{M^{t}})=\{r: r\geqslant t\}$. First, we consider the case that $t$ is a successor of some element $s$. Since $s$ is the minimal element in $\Ass(\widetilde{M^s})$, by Lemma \ref{Lemma:key}, $\Ass(\widetilde{M^{t}})=\{r: r\geqslant t\}$ and $\Ass(\widetilde{M^s/M^{t}})=\{s\}$. Secondly, we assume that $t$ is a limit point. Note that for any $r<t$, $\Ass(\widetilde{M^{t}})\subseteq \Ass(\widetilde{M^r})=\{s: s\geqslant r\}$, so $\Ass(\widetilde{M^t})\subseteq \bigcap\limits_{r<t}\{s: s\geqslant r\}=\{s: s\geqslant t\}$. We want to show that for any $s$ such that $s\geqslant t$, $s\in \Ass(\widetilde{M^t})$. Note that $s\in \Ass(\widetilde{M})$, by Proposition \ref{prop:ann}, there exists some $x\in M$ such that $\ann_A(x)=s$. Let $N$ be the submodule of $M$ generated by $x$. Consider the $A$-module homomorphism $x: A\rightarrow N$ which maps $1$ to $x$, we have a short exact sequence
\[\xymatrix{0\ar[r]& s\ar[r]&A\ar[r]& N\ar[r]& 0},\]
 thus $N\cong A/s$. We claim that for any $r$ such that $r<t$, $N\subseteq M^r$. We prove the claim by transfinite induction. When $r=0$, $N\subseteq M$ naturally holds. Now, we assume that for any $q$ such that $q<r$, $N\subseteq M^q$. We want to prove $N\subseteq M^r$. First, we consider the case that $r$ is a successor of some element $q$. Let $\phi_q$ be the homomorphism $M^q\rightarrow (M^q)_q$. Since $N\subseteq M^q$, we have $x\in M^q$. Because both $\ann_A(x)=r$ and $r\not\subseteq q$ hold, there exists some $a\in r\backslash q$ such that $ax=0$. Therefore, $\phi_q(x)=0$ in $(M^q)_q$. Since $M^r=\ker(\phi_q)$, we have $x\in M^r$, thus $N\subseteq M^r$. Secondly, we assume that $r$ is a limit point. Since $N\subseteq M^q$ for any $q<r$, we have $N\subseteq \bigcap\limits_{q<r} M^q=M^r$. Therefore, for any $r<t$, $N\subseteq M^r$. Since $M^t=\bigcap\limits_{r<t} M^r$, we have $N\subseteq M^t$. Since $N\cong A/s$, we have $\Ass(\widetilde N)=\{s\}$. Since $\Ass(\widetilde N)\subseteq \Ass(\widetilde {M^t})$, $s\in \Ass(\widetilde {M^t})$. Therefore, for any $s$ such that $s\geqslant t$, $s\in \Ass(\widetilde{M^t})$. Hence, $\Ass(M^t)=\{s:s\geqslant t\}$.\\

Step 3: Uniqueness. \\
 We assume $N^t$ also satisfies the properties. We prove $N^t=M^t$ by transfinite induction. If the set $S=\{r\in \Ass(\widetilde M): M^r\neq N^r\}$ is non-empty, we may take $t$ to be the minimal element of $S$. Then $t\neq 0$ and $t$ is not a limit point. Now we may assume $t=r+1$. By the definition of $t$, we have $M^r=N^r$. Note that both $\Ass(\widetilde{N^r/N^t})=\{r\}$ and $\Ass(\widetilde{N^t})=\{q: q\geqslant t\}=\Ass(\widetilde{N^r})\backslash \{r\}$ hold.
 By Lemma \ref{Lemma:key}, we can see that $N^t=\ker(N^r\rightarrow (N^r)_r)=M^t$.\\

Step 4: Properties.\\
We only need to prove if $(\Ass(\widetilde{M}),<)$ is isomorphic to a limit ordinal, $\bigcap M^r=0$. Note that
$\Ass(\widetilde{\bigcap  M^r})\subseteq \Ass(\widetilde {M^t})=\{s: s\geqslant t\}$ for any $t\in \Ass(\widetilde M)$. Thus, $\Ass(\widetilde{\bigcap M^r})=\emptyset$, so $\bigcap M^r=0$.

\end{proof}

\begin{eg}
In this example, we construct a filtration of $\mathbb Z$-modules. Let $\mathcal P$ be the set of primes. Let $M=\mathbb Z\bigoplus \mathbb Z\bigoplus (\bigoplus\limits_{p\in \mathcal P} \mathbb Z/p\mathbb Z)$, $M^1=\mathbb Z\bigoplus (\bigoplus\limits_{p\in \mathcal P} \mathbb Z/p\mathbb Z)$, $M^2=2\mathbb Z\bigoplus (\bigoplus\limits_{p\geqslant 3} \mathbb Z/p\mathbb Z)$, $M^3=6\mathbb Z\bigoplus (\bigoplus\limits_{p\geqslant 5} \mathbb Z/p\mathbb Z)$, $\cdots$. We can see that $\bigcap M^r=0$.
Therefore, condition \ref{Item:Important} in Theorem \ref{Thm:commutative ring} is necessary for the uniqueness.
\end{eg}

\section{Equivalent coprimary filtrations}
\begin{prop}\label{Prop:intersection}
Let $X$ be a locally Noetherian scheme, $\mathcal F$ be a non-zero coherent sheaf on $X$ and $(\Ass(\mathcal F),<)$ be the well-founded partially ordered set defined in Example \ref{Eg: key}. Let $i:\Spec \mathcal O_{X,\mathfrak{p}}\rightarrow X$ and $j:\Spec \mathcal O_{X,\mathfrak{q}}\rightarrow X$ be canonical morphisms, where $\mathfrak{p},\mathfrak{q}\in \Ass(\mathcal F)$, $\mathfrak{p}\neq \mathfrak{q}$ and $\rk(\mathfrak{p})=\rk(\mathfrak{q})=0$. Denote
$\ker(\mathcal F\rightarrow i_*i^*\mathcal F)$ by $\mathcal G$, $\ker(\mathcal F\rightarrow j_*j^* \mathcal F)$ by $\mathcal H$,
then \[\mathcal G\cap \mathcal H=\ker(\mathcal G\rightarrow j_*j^*\mathcal G)=\ker(\mathcal H\rightarrow i_*i^*\mathcal H).\]
\end{prop}
\begin{proof}
We only need to show $\mathcal G\cap \mathcal H=\ker(\mathcal G\rightarrow j_*j^*\mathcal G)$ by symmetry. Consider the following diagram:
\[\xymatrix{\mathcal G\cap \mathcal H\ar@{}[rd]|-{\square}\ar[r]^-{f_1}\ar[d]_-{f_2}&\mathcal G\ar[r]^-{\varphi_1}\ar[d]_-{g}&j_*j^*\mathcal G\ar[d]^-{g'}\\
\mathcal H\ar[r]_-{h}&\mathcal F\ar[r]_-{\varphi_2}&j_*j^*\mathcal F}\]
We can see that $g'\circ \varphi_1\circ f_1=\varphi_2\circ h\circ f_2=0$.
Since $g:\mathcal G\rightarrow \mathcal F$ is injective, we have an injective $\mathcal O_{X,\mathfrak{q}}$ -module homomorphism $\mathcal G_\mathfrak{q}\rightarrow \mathcal F_\mathfrak{q}$ induced by $g$, hence $g':j_*j^*\mathcal G\rightarrow j_*j^*\mathcal F$ is injective. Because $g'$ is a monomorphism,
$\varphi_1\circ f_1=0$. Therefore,
\[\mathcal G\cap \mathcal H\subseteq \ker(\mathcal G\rightarrow j_*j^*\mathcal G).\]
Denote $\ker(\mathcal G\rightarrow j_*j^*\mathcal G)$ by $i:\mathcal W\rightarrow \mathcal G$.
Hence, $\varphi_2\circ g\circ i=g'\circ \varphi_1\circ i=0$.
Since $\mathcal H=\ker(\varphi_2)$, we can see that
$\ker(\mathcal G\rightarrow j_*j^*\mathcal G)\subseteq \mathcal H$. Therefore, \[\ker(\mathcal G\rightarrow j_*j^*\mathcal G)\subseteq \mathcal G\cap \mathcal H.\]

\end{proof}

\begin{cor}\label{cor: filtration}Let $X$ be a Noetherian scheme, $\mathcal F$ be a non-zero  coherent sheaf on $X$. Let
\[0=\mathcal F^0\subseteq \mathcal F^1\subseteq \cdots \subseteq \mathcal F^{i-2}\subseteq \mathcal F^{i-1}\subseteq \mathcal F^i\subseteq \cdots \subseteq \mathcal F^n=\mathcal F\]
be a coprimary filtration of $\mathcal  F$, and let $\Ass(\mathcal F^i/\mathcal F^{i-1})=\{\mathfrak{p}\}$,
$\Ass(\mathcal F^{i-1}/\mathcal F^{i-2})=\{\mathfrak{q}\}$, $\mathcal G^{i-1}=\ker(\mathcal F^i\rightarrow  i_*i^*\mathcal F^i)$, where $i:\Spec O_{X,\mathfrak{q}}\rightarrow X$ is a canonical morphism. Assume that $\overline{\{\mathfrak{q}\}}\not \subseteq \overline{\{\mathfrak{p}\}}$,  then the filtration
\[0=\mathcal F^0\subseteq \mathcal F^1\subseteq \cdots \subseteq \mathcal F^{i-2}\subseteq \mathcal G^{i-1}\subseteq \mathcal F^i\subseteq \cdots\subseteq \mathcal F^n=\mathcal F\]
is also a coprimary filtration of $\mathcal F$.
\end{cor}
\begin{proof}
Because $\overline{\{\mathfrak{q}\}}\not \subseteq \overline{\{\mathfrak{p}\}}$, $\rk(\mathfrak{p})=\rk(\mathfrak{q})=0$ in $\Ass(\mathcal F^i, <)$. By Proposition \ref{Prop:intersection}, \[\mathcal F^{i-1}\cap \mathcal G^{i-1}=\mathcal F^{i-2}=\ker(\mathcal G^{i-1}\rightarrow i_*i^*\mathcal G^{i-1}).\]
Therefore, by Theorem \ref{Thm: coprimary filtration}, the filtration
\[0=\mathcal F^0\subseteq \mathcal F^1\subseteq \cdots \subseteq \mathcal F^{i-2}\subseteq \mathcal G^{i-1}\subseteq \mathcal F^i\subseteq \cdots\subseteq \mathcal F^n=\mathcal F\]
is also a coprimary filtration of $\mathcal F$.
\end{proof}

\begin{prop}\label{Prop:sum}
Let $X$ be a locally Noetherian scheme, $\mathcal F$ be a non-zero coherent sheaf on $X$, and $(\Ass(\mathcal F),<)$ be the well-founded partially ordered set defined in Example \ref{Eg: key}. Let $i:\Spec \mathcal O_{X,\mathfrak{p}}\rightarrow X$ and $j:\Spec \mathcal O_{X,\mathfrak{q}}\rightarrow X$ be canonical morphisms, where $\mathfrak{p},\mathfrak{q}\in \Ass(\mathcal F)$, $\mathfrak{p}\neq \mathfrak{q}$ and $\rk(\mathfrak{p})=\rk(\mathfrak{q})=0$.
Denote
$\ker(\mathcal F\rightarrow i_*i^*\mathcal F)$ by $\mathcal G$, $\ker(\mathcal F\rightarrow j_*j^* \mathcal F)$ by $\mathcal H$.
If $\overline{\{\mathfrak{p}\}}\cap \overline{\{\mathfrak{q}\}}=\emptyset$, then $\mathcal F=\mathcal G+\mathcal H$.
\end{prop}
\begin{proof}
Note that
\[\Supp(\mathcal F/(\mathcal G+\mathcal H))\subseteq \Supp(\mathcal F/\mathcal H)=\overline{\{\mathfrak{q}\}}.\]
Similarly, $\Supp(F/(\mathcal G+\mathcal H))\subseteq \overline{\{\mathfrak{p}\}}$. Since $\overline{\{\mathfrak{p}\}}\cap \overline{\{\mathfrak{q}\}}=\emptyset$,
$\Supp(\mathcal F/(\mathcal G+\mathcal H))=\emptyset$. Thus $\mathcal F/(\mathcal G+\mathcal H)=0$. Therefore, $\mathcal F=\mathcal G+\mathcal H$.

\end{proof}

The following example shows that the condition $\overline{\{\mathfrak{p}\}}\cap \overline{\{\mathfrak{q}\}}=\emptyset$ is necessary in Proposition \ref{Prop:sum}.
\begin{eg}
Let $A=\mathbb{C}[x,y]$, $M=\mathbb{C}[x,y]/(xy)$.
We can see that $\Ass(M)=\{xA,yA\}$. Let $L=\ker(M\rightarrow M_{yA})$,
$N=\ker(M\rightarrow M_{xA})$. Thus,
$L=yM$ and $N=xM$. However, since $1\notin L+N$,
 $L+N\neq M$.
\end{eg}

\begin{defn}\label{defn:equivalence}
Let $X$ be a Noetherian scheme, $\mathcal F$ be a non-zero coherent sheaf on $X$.
Let
\[0=\mathcal F^0\subseteq \mathcal F^1\subseteq \cdots \subseteq \mathcal F^n=\mathcal F\]
and
\[0=\mathcal G^0\subseteq \mathcal G^1\subseteq \cdots \subseteq \mathcal G^n=\mathcal F\]
be two coprimary filtrations of $\mathcal F$. We say they are equivalent if $\mathcal F^{i+1}/\mathcal F^i\cong \mathcal G^{j+1}/\mathcal G^j$ when $\Ass(\mathcal F^{i+1}/\mathcal F^i)=\Ass(\mathcal G^{j+1}/\mathcal G^j)$.
\end{defn}

It can be shown that Definition \ref{defn:equivalence} defines an equivalence relation on the set of coprimary filtrations of $\mathcal F$.

\begin{thm}
Let $X$ be a  Noetherian scheme, $\mathcal F$ be a non-zero coherent sheaf on $X$ and $(\Ass(\mathcal F),<)$ be the well-founded partially ordered set defined in Example \ref{Eg: key}.
Assume for any $\mathfrak{p},\mathfrak{q}\in \Ass(\mathcal F)$ satisfying both $\mathfrak{p}\notin \overline{\{\mathfrak{q}\}}$ and $\mathfrak{q}\notin \overline{\{\mathfrak{p}\}}$,
we always have $\overline{\{\mathfrak{p}\}}\cap \overline{\{\mathfrak{q}\}}=\emptyset$.
Then any two coprimary filtrations of $\mathcal F$  are equivalent.
\end{thm}
\begin{proof}
Let $\Ass(\mathcal F)=\{\mathfrak{p}_1,\cdots,\mathfrak{p}_n\}$.
We prove the theorem by induction on $n$. When $n=1$, the theorem holds naturally.
Now, we assume that when $n=m-1$, this theorem holds. We assume that when $n=m$, the theorem does not hold.
Thus, there exist two coprimary filtrations of $\mathcal F$,  \begin{equation}\label{equ:first filtration}
0=\mathcal F^0\subseteq \mathcal F^1\subseteq \cdots \subseteq \mathcal F^m=\mathcal F
\end{equation}
and\begin{equation}\label{equ:second filtration}
0=\mathcal G^0\subseteq \mathcal G^1\subseteq \cdots \subseteq \mathcal G^m=\mathcal F,
\end{equation}
such that they are not equivalent. Let $\Ass(\mathcal F^{i+1}/\mathcal F^i)=\{\mathfrak{p}_{i+1}\}$,
$\Ass(\mathcal G^{j+1}/\mathcal G^j)=\{\mathfrak{q}_{j+1}\}$, where $i,j\in \{0,1,\cdots,n-1\}$.
If $\mathfrak{p}_m=\mathfrak{q}_m$, then by Theorem \ref{Thm: coprimary filtration}, $\mathcal F^{m-1}=\mathcal G^{m-1}=\ker(\mathcal F\rightarrow i_*i^*\mathcal F)$, where $i:\Spec(\mathcal O_{X,\mathfrak{p}_m})\rightarrow X$ is a canonical morphism. Using the induction hypothesis on $\mathcal F^{m-1}$, we are done. Now, we may assume
$\mathfrak{p}_m\neq \mathfrak{q}_m$ and $\mathfrak{p}_m=\mathfrak{q}_i$, where $i\neq m$. We could also assume that there does not exist a coprimary filtration  of $\mathcal F$, \begin{equation}\label{equ:third filtration}
0=\mathcal G'^0\subseteq \mathcal G'^1\subseteq \cdots \subseteq \mathcal G'^m=\mathcal F,
\end{equation}
satisfying the following conditions,
\begin{enumerate}
\item the coprimary filtration (\ref{equ:third filtration}) is not equivalent to the coprimary filtration (\ref{equ:first filtration}),
\item  $\Ass(\mathcal G'_k/\mathcal G'_{k-1})=\mathfrak{p}_m$ for some $k>i$.
\end{enumerate}
 Note that
$\mathfrak{q}_{i}=\mathfrak{p}_m\notin \overline{\{\mathfrak{q}_{i+1}\}}$, by Corollary \ref{cor: filtration}, there exists a coprimary filtration
\begin{equation}\label{equ:fourth filtration}
0=\mathcal G^0\subseteq \mathcal G^1\subseteq \cdots \subseteq \mathcal G^{i-1}\subseteq \mathcal F'^{i}\subseteq \mathcal G^{i+1}\subseteq \cdots\subseteq \mathcal G^n=\mathcal F,
\end{equation}
such that both $\Ass(\mathcal G^{i+1}/\mathcal F'^i)=\{\mathfrak{q}_i\}$ and $\Ass(\mathcal F'^i/\mathcal G^{i-1})=\{\mathfrak{q}_{i+1}\}$ hold, where
$\mathcal F'^i=\ker(\mathcal G^{i+1}\rightarrow i_*i^*\mathcal G^{i+1})$. By Proposition \ref{Prop:intersection}, $\mathcal G^{i-1}=\mathcal G^i\cap \mathcal F'^i$. Since $\overline{ \{\mathfrak{q}_{i}\}}\cap \overline{\{\mathfrak{q}_{i+1}\}}=\emptyset$, by Proposition \ref{Prop:sum}, $\mathcal G^{i+1}=\mathcal F'^i+\mathcal G^i$. Hence, $\mathcal G^{i+1}/\mathcal F'^i\cong \mathcal G^i/\mathcal G^{i-1}$ and
$\mathcal G^{i+1}/\mathcal G_i\cong \mathcal F'^i/\mathcal G^{i-1}$. Therefore,
the  coprimary filtration (\ref{equ:fourth filtration}) is equivalent to the coprimary filtration (\ref{equ:second filtration}). Thus, the  coprimary filtration (\ref{equ:fourth filtration}) is not equivalent to the coprimary filtration (\ref{equ:first filtration}), which leads to a contradiction since $\Ass(\mathcal G^{i+1}/\mathcal F'^i)=\{\mathfrak{p}_m\}$.

\end{proof}

\begin{cor}
Let $A$ be a Dedekind domain, $M$ be a non-zero finitely generated module over $A$. Then any two coprimary filtrations of $M$ are equivalent.
\end{cor}
\begin{proof}
Let $\mathfrak{p},\mathfrak{q}\in \Ass(M)$ satisfying both $\mathfrak{p}\not \subseteq \mathfrak{q}$ and $\mathfrak{q}\not \subseteq \mathfrak{p}$. Then both $\mathfrak{p}$ and $\mathfrak{q}$ are different maximal ideals of $A$, so $\overline{\{\mathfrak{p}\}}\bigcap \overline{\{\mathfrak{q}\}}=\emptyset$.
\end{proof}

\section{Direct sum decomposition}

\begin{thm}\label{thm:direct sum}
Let $X$ be a  Noetherian scheme and $\mathcal F$ be a non-zero coherent sheaf on $X$.
 Denote $\Ass(\mathcal F)$ by $\{\mathfrak{p}_1,\cdots, \mathfrak{p}_n\}$. Assume that $\overline{\{\mathfrak{p}_i\}}\cap \overline{\{\mathfrak{p}_j\}}=\emptyset$, where $i,j\in \{1,\cdots, n\}$ and $i\neq j$. Denote the canonical morphisms $\Spec(\mathcal O_{X,\mathfrak{p}_j})\rightarrow X$ by $t_j$, where $j\in \{1,\cdots, n\}$. Let
\[\mathcal F_i=\bigcap_{j\neq i}\ker(\mathcal F\rightarrow t_{j*}t_j^*\mathcal F),\]
 $0=\mathcal G^0\subseteq \mathcal G^1\subseteq \cdots \subseteq \mathcal G^n=\mathcal F$ be a coprimary filtration of $\mathcal F$, and
  $\Ass(\mathcal G^1)=\{\mathfrak{p}_i\}$. Then $\mathcal F_i=\mathcal G^1$. Moreover,
  \[\mathcal F=\bigoplus_{1\leqslant i\leqslant n} \mathcal F_i.\]
\end{thm}
\begin{proof}
Without loss of generality, we may assume that $\Ass(N_{i+1}/N_i)=\{\mathfrak{p}_{i+1}\}$, where $i\in \{0,\cdots,n-1\}$. We claim that for any $i\in \{0,\cdots,n-1\}$,
\[\mathcal G^i=\bigcap_{j>i} \ker(\mathcal F\rightarrow t_{j*}t_j^*\mathcal F).\]
 By Theorem \ref{Thm: coprimary filtration}, the claim is true when $i=n-1$. We prove the claim by induction.
 We assume that the claim is true when $i=k+1$. Let $i=k$.
 Since $ \mathfrak{p}_{k+1}\notin \overline{\{\mathfrak{p}_{k+2}\}}$, by Corollary \ref{cor: filtration}, there exists a coprimary filtration of $\mathcal F$,
 \[0=\mathcal G^0\subseteq \mathcal G^1\subseteq \cdots \subseteq \mathcal G^{k}\subseteq \mathcal G'^{k+1}\subseteq \mathcal G^{k+2}\subseteq \cdots\subseteq \mathcal G^n=\mathcal F,\]
 such that both $\Ass(\mathcal G^{k+2}/\mathcal G'^{k+1})=\{\mathfrak{p}_{k+1}\}$ and $\Ass(\mathcal G'^{k+1}/\mathcal G^{k})=\{\mathfrak{p}_{k+2}\}$ hold, where
$\mathcal G'^{k+1}=\ker(\mathcal G^{k+2}\rightarrow t_{k+1 *}t^*_{k+1}\mathcal G^{k+2})$. By Proposition \ref{Prop:intersection},
$\mathcal G^k=\mathcal G^{k+1}\cap \mathcal G'^{k+1}$. By induction hypothesis, we have
\[\mathcal G^{k+1}=\bigcap_{j>k+1} \ker(\mathcal F\rightarrow t_{j*}t_j^*\mathcal F)\]
and
\[\mathcal G'^{k+1}=(\bigcap_{j>k+2} \ker(\mathcal F\rightarrow t_{j*}t_j^*\mathcal F))\cap \ker(\mathcal F\rightarrow t_{k+1 *}t_{k+1}^*\mathcal F).\]
Therefore,
\[\mathcal G^k=\mathcal G^{k+1}\cap \mathcal G'^{k+1}=\bigcap_{j>k} \ker(\mathcal F\rightarrow t_{j*}t_j^*\mathcal F).\]
In particular, if $i=1$,
\[\mathcal G^1=\bigcap_{j>1} \ker(\mathcal F\rightarrow t_{j*}t_j^*\mathcal F)=\mathcal F_1.\]
Now, we claim that for any $i\in \{1,\cdots,n\}$,
\[\mathcal G^i=\sum_{j\leqslant i} \mathcal F_j.\]
The claim is true when $i=1$. We prove the claim by induction.
We assume that the claim is true when $i=k$. Let $i=k+1$. Since $ \mathfrak{p}_{k}\notin \overline{\{\mathfrak{p}_{k+1}\}}$, by Corollary \ref{cor: filtration}, there exists a coprimary filtration of $\mathcal F$,
 \[0=\mathcal G^0\subseteq \mathcal G^1\subseteq \cdots \subseteq \mathcal G^{k-1}\subseteq \mathcal G'^{k}\subseteq \mathcal G^{k+1}\subseteq \cdots\subseteq \mathcal G^n=\mathcal F,\]
 such that both $\Ass(\mathcal G^{k+1}/\mathcal G'^{k})=\{\mathfrak{p}_{k}\}$ and $\Ass(\mathcal G'^{k}/\mathcal G^{k-1})=\{\mathfrak{p}_{k+1}\}$ hold, where
$\mathcal G'^{k}=\ker(\mathcal G^{k+1}\rightarrow t_{k*}t_k^* \mathcal G^{k+1})$. Since $\overline{\{\mathfrak{p}_k\}}\cap \overline{\{\mathfrak{p}_{k+1}\}}=\emptyset$, by Proposition \ref{Prop:sum},
$\mathcal G^{k+1}=\mathcal G^k+\mathcal G'^{k}$. By induction hypothesis, both
$\mathcal G^{k}=\sum\limits_{j\leqslant k} \mathcal F_j$
and
$\mathcal G'^{k}=\sum\limits_{j\leqslant k-1}\mathcal F_j+\mathcal F_{k+1}$ hold.
Therefore,
\[\mathcal G^{k+1}=\mathcal G^{k}+ \mathcal G'^{k}=\sum_{j\leqslant k+1} \mathcal F_j.\]
In particular, if $i=n$,
\[\mathcal F=\sum_{1\leqslant j\leqslant n} \mathcal F_j.\]
To show such a sum is a direct sum, without loss of generality, we only need to show
\[\mathcal H:=\mathcal F_n\cap (\sum_{1\leqslant i\leqslant n-1} \mathcal F_i)=0.\]
Note that $\mathcal H$ is a coherent sheaf on $X$, $\Ass(\mathcal H)\subseteq \Ass(\mathcal F_n)=\{p_n\}$ by Theorem \ref{Thm: coprimary filtration}.
Moreover,
\[\Ass(\mathcal H)\subseteq \Ass(\sum_{1\leqslant i\leqslant n-1} \mathcal F_i)=\Ass(\mathcal G^{n-1})=\{\mathfrak{p}_1,\cdots,\mathfrak{p}_{n-1}\}.\]
Thus, $\Ass(\mathcal H)=\emptyset$. Therefore, $\mathcal H=0$.

\end{proof}

\begin{rmk}
In Theorem \ref{thm:direct sum}, if $X=\Spec(A)$ and $\Ass(M)$ consists of maximal ideals of $A$, then $\overline{\{\mathfrak{p}_i\}}\bigcap \overline{\{\mathfrak{p}_j\}}=\emptyset$ holds, where $i,j\in \{1,\cdots, n\}$ and $i\neq j$. Therefore, Kirby's result \cite[Theorem 3]{K1973} is a corollary of Theorem \ref{thm:direct sum}.
\end{rmk}

\bibliographystyle{plain}
\bibliography{refer}

\end{document}